\documentclass[11pt]{amsart}
\usepackage{amsmath}
\usepackage{amsthm}
\usepackage{amsfonts}
\usepackage{graphicx}

\newtheorem{theorem}{Theorem}[section]

\theoremstyle{definition}

\theoremstyle{remark}

\numberwithin{equation}{section}
\newcommand{\cone}{\mbox{$\times \hspace*{-0.258cm} \times$}}

\begin{document}

\title[Higher dimensional Schwarz's surfaces  and Scherk's surfaces]
{Higher dimensional Schwarz's surfaces  and Scherk's surfaces}
\author[J. CHOE and J. Hoppe ]{JAIGYOUNG CHOE and JENS HOPPE}
\thanks{J.C. supported in part by NRF 2011-0030044, SRC-GAIA}

\address{Korea Institute for Advanced Study, Seoul, 02455, Korea}
\email{choe@kias.re.kr}

\address{KTH, 100 44 Stockholm,
Sweden\\}
\email{hoppe@kth.se}

\begin{abstract}
Higher dimensional generalizations of Schwarz's $P$-surface, Schwarz's $D$-surface and Scherk's second surface are constructed as complete embedded periodic minimal hypersurfaces in $\mathbb R^n$.
\end{abstract}

\maketitle

In $\mathbb R^3$ minimal surfaces are easy to construct. Thanks to the existence of isothermal coordinates on a surface, one can derive the Weierstrass representation formula, which allows one to obtain minimal surfaces in $\mathbb R^3$ at will. Nonetheless, only a few topologically simple complete minimal surfaces were known to exist in $\mathbb R^3$ until recently.

It is not easy to understand the topology of a minimal surface in terms of its Weierstrass data. It is ironical that many of these well-known simple minimal surfaces could be constructed without resorting to the Weierstrass representation formula. The catenoid, the helicoid, Enneper's surface, Scherk's first surface, Scherk's second surface, Schwarz's $P$-surface and Schwarz's $D$-surface can be constructed by exploiting their geometric characteristics.

In $\mathbb R^n,n\geq4$, there is no systematic method to construct minimal hypersurfaces. So far, only the catenoid \cite{B}, the helicoid \cite{CH} and Enneper's surface \cite{C} are known to have higher dimensional versions in $\mathbb R^n$. In this paper we construct the higher dimensional generalizations of Schwarz's $P$-surface, Schwarz's $D$-surface and Scherk's second surface. First, we extract geometric characteristics of their fundamental pieces, and then solve the Dirichlet problem to construct the higher dimensional versions of the fundamental pieces and extend them across their boundaries by $180^\circ$-rotation.

\section{Schwarz's $P$-surface}
A triply periodic minimal surface in $\mathbb R^3$ was first constructed by H.A. Schwarz \cite{S} in 1865. It was found as a by-product in the process of solving the Plateau problem in a concrete case. The Jordan curve that Schwarz considered was the skew quadrilateral $\Gamma$ consisting of the four edges of a regular tetrahedron $T$. He found the minimal surface $S_0$ spanning $\Gamma$ from explicit data for the Weierstrass representation formula. Since $T$ fits nicely in a cube $Q^3$ so that each edge of $\Gamma$ becomes a diagonal on the square faces of $Q^3$, Schwarz was able to show that the analytic extension $S$ of $S_0$ is an embedded triply periodic minimal surface in $\mathbb R^3$. He also proved that $S^*$, the conjugate minimal surface of $S$, is embedded and triply periodic as well. The quadrilateral $\Gamma^*$ bounding the fundamental piece of $S^*$ has vertex angles of $\pi/3,\pi/3,\pi/2,\pi/2$ while those of $\Gamma$ are $\pi/3,\pi/3,\pi/3,\pi/3$. Because of the vertex angles $\pi/2,\pi/2$ of $\Gamma^*$ $S^*$ turns out to be perpendicular to $\partial Q^3$. Moreover, due to the vertex angles $\pi/3,\pi/3$ of $\Gamma^*$ as well as $\Gamma$, both $S$ and $S^*$ contain three straight lines meeting at every flat point. In fact $S^*$ is the well-known Schwarz $P$-surface. Figure 1 shows a fundamental piece of $S^*$ in $[-3,1]\times[-3,1]\times[-3,1]$:
\begin{center}
\includegraphics[width=2.5in]{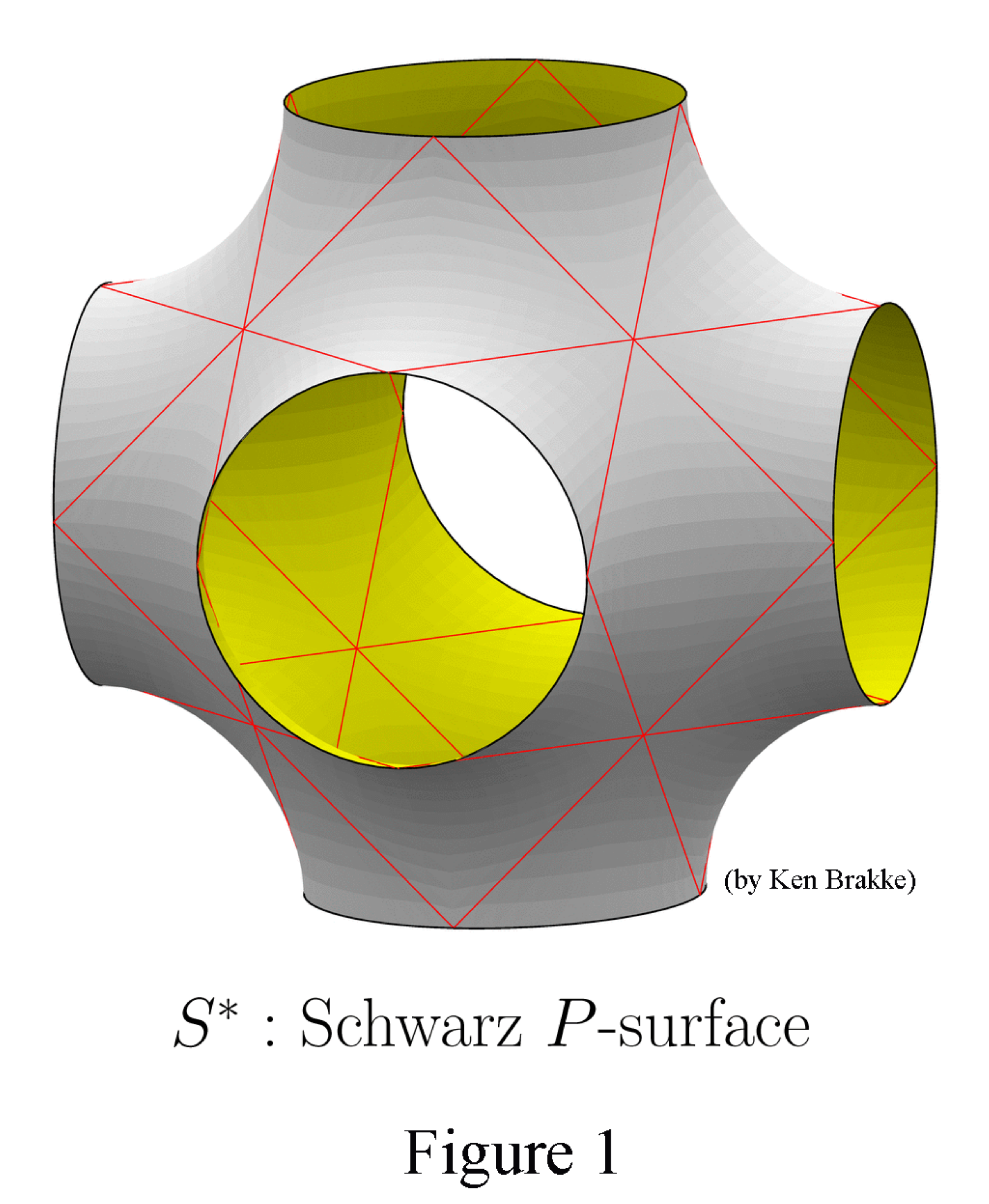}\\
\end{center}

Part of $S^*$ in a smaller cube $Q^3$ is shown in Figure 2. This part, denoted $H$, is diffeomorphic to a hexagon and consists of 6 congruent triangular pieces. Each triangular piece is bounded by two line segments and a planar curve. Along this curve the triangular piece is perpendicular to the face of the cube. $H$ is close to the regular hexagon $H_0$ such that $L:=H\cap H_0$ is three straight lines meeting each other at $60^\circ$. Let's introduce a coordinate system $(x_1,x_2,x_3)$ such that $$Q^3=[-1,1]\times[-1,1]\times[-1,1]\,\,\,\,{\rm and}\,\,\,\, H_0=\{(x_1,x_2,x_3)\in Q^3:x_1+x_2+x_3=0\}.$$ Then the three straight lines $L$ in $H$ are the intersection of $H_0$ with the three coordinate planes of $\mathbb R^3$ and furthermore
$$L=H_0\cap\{(x_1,x_2,x_3):x_1+x_2=0 \,\,\,\,{\rm or}\,\,\,\,x_2+x_3=0 \,\,\,\,{\rm or}\,\,\,\,x_1+x_3=0\}.$$
\begin{center}
\includegraphics[width=2.1in]{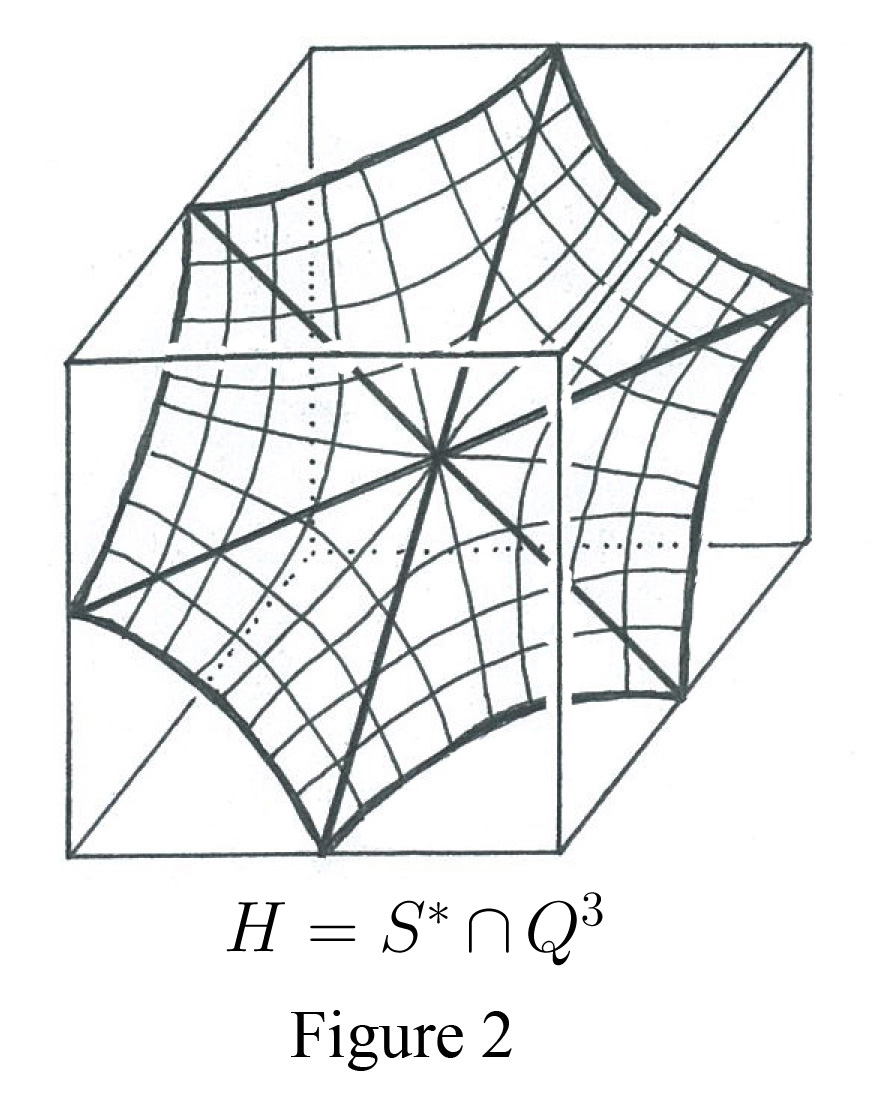}\\
\end{center}

We will generalize these properties of $S^*\cap Q^3$ to find a higher dimensional Schwarz surface in $\mathbb R^n$. First let $Q^n$ be the $n$-dimensional cube in $\mathbb R^n$
$$Q^n=[-1,1]^n=\{(x_1,\ldots,x_n):-1\leq x_i\leq1\}.$$
Define
$$P_n=\{(x_1,\ldots,x_n)\in Q^n:x_1+\cdots+x_n=0\}.$$
Then $P_n$ is an $(n-1)$-dimensional polyhedron with $2n$ faces, that is,
$$\partial P_n=\left(\bigcup_{i=1}^nB_i^+\right)\,\cup\,\left(\bigcup_{i=1}^nB_i^-\right),$$
where
$$B_i^+=\{(x_1,\ldots,x_n)\in\partial Q^n:x_i=1,x_1+\cdots+\widehat{x}_i+\cdots+x_n=-1\},$$
$$B_i^-=\{(x_1,\ldots,x_n)\in\partial Q^n:x_i=-1,x_1+\cdots+\widehat{x}_i+\cdots+x_n=1\}.$$
\begin{center}
\includegraphics[width=2.9in]{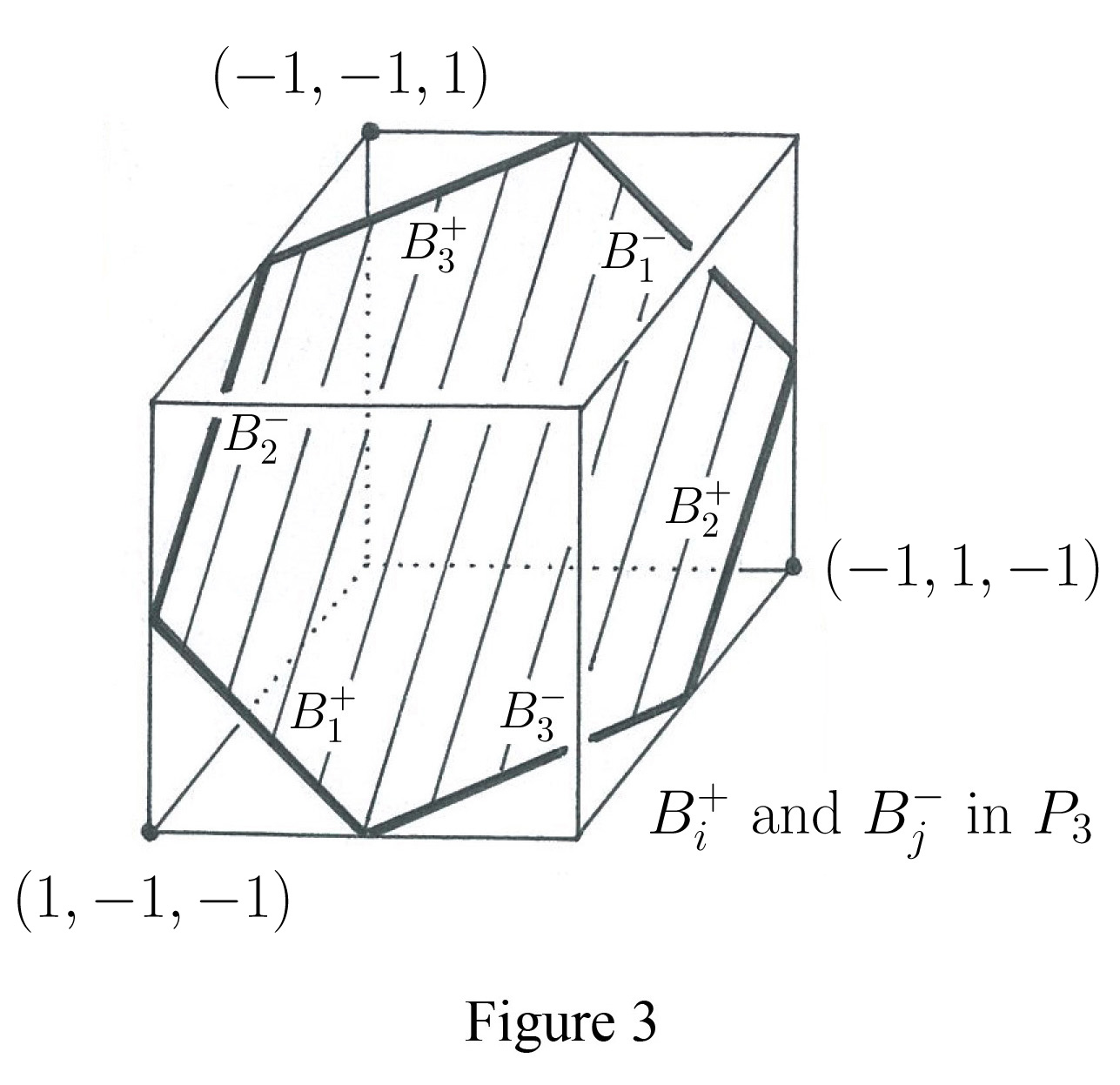}\\
\end{center}
$P_3$ is a regular hexagon and $P_4$ is a regular octahedron. What can one say about $P_n$? Let $G_1\subset O(n)$ be the group of all isometries of $\mathbb R^n$ which act on $\{x_1,\ldots,x_n\}$ as permutations and define $\varphi:\mathbb R^n\rightarrow\mathbb R^n$ by $\varphi(x)=-x,\,x\in\mathbb R^n$. Let $G_2$ be the subgroup of $O(n)$ generated by $G_1\cup\{\varphi\}$. Then for any $B_i^+$ and $B_j^-$ there exist isometries $\psi_1\in G_1$ and $\psi_2\in G_2$ such that
$$\psi_1(B_1^+)=B_i^+\,\,\,\,{\rm and}\,\,\,\,\psi_2(B_1^+)=B_j^-.$$
Therefore one can say that the faces of $P_n$ are congruent to each other.

More precisely,
\begin{eqnarray*}
&B_n^+&
=\{x_n=1,\,x_1+\cdots+x_{n-1}=-1\}\cap Q^n\\
&=&\{x_n=1,\,x_1+\cdots+x_{n-1}=-1\}\cap\{-1\leq x_1,\ldots,x_{n-1}\}\cap\{x_1,\ldots,x_{n-1}\leq1\}\\
&=&\{x_n=1,\,x_1+\cdots+x_{n-1}=-1\}\cap\{-1\leq x_1,\ldots,x_{n-1}\leq n-3\}\cap\{x_1,\ldots,x_{n-1}\leq 1\}\\
&=&A\cap\{x_1,\ldots,x_{n-1}\leq 1\},
\end{eqnarray*}
where $A:=\{x_n=1,\,x_1+\cdots+x_{n-1}=-1\}\cap\{-1\leq x_1,\ldots,x_{n-1}\leq n-3\}$ is the regular $(n-2)$-simplex with vertices $(n-3,-1,\ldots,-1,1),\,(-1,n-3,-1,\ldots,-1,1),\ldots,$ $(-1,\ldots,-1,n-3,1)$. Then $B_n^+$ is the truncated regular $(n-2)$-simplex, i.e., truncated by the half spaces $\{1< x_i\},\,i=1,\ldots,n-1$, at all its vertices. In case $n=3$ and $4$, $B_i^\pm$ is the regular $(n-2)$-simplex with no truncation.
\begin{center}
\includegraphics[width=4.5in]{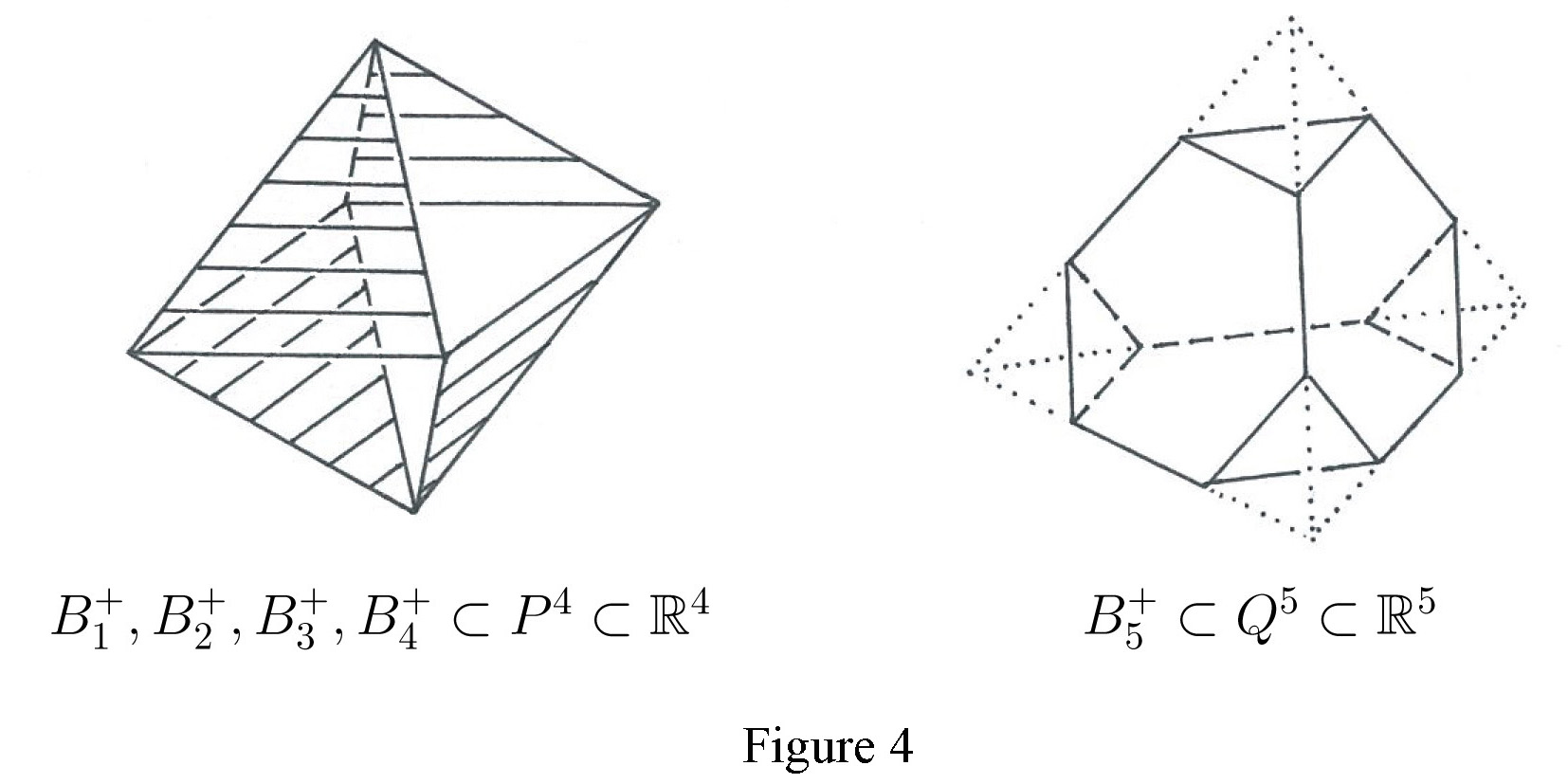}\\
\end{center}

In dimension $n\geq5$, the faces of $B_i^\pm$ consist of the faces of the regular $(n-2)$-simplex and those created by the truncation. In other words,
$$\partial B_i^\pm=\left(\bigcup_{j=1}^{n-1} F_j\right)\cup\left(\bigcup_{j=1}^{n-1} \hat{F}_j\right),$$ where $F_j$ is a subset of a face of the $(n-2)$-simplex and $\hat{F}_j$ is the face created by the truncation at each vertex. In dimension $n=3,4$, however, $\partial B_i^\pm=\bigcup_{j=1}^{n-1} F_j$.

The three straight lines $L= H\cap H_0$ mentioned above is called the {\it spine} of $H$ (or of $H_0$). The { spine} $L_n$ of $P_n$ is defined as
$$L_n=\left(\bigcup_{i=1}^nO\cone\partial B_i^+\right)\,\cup\,\left(\bigcup_{i=1}^nO\cone\partial B_i^-\right),$$
where $O$ is the origin of $\mathbb R^n$ and $O\cone\partial B_i^{\pm}$ denotes the cone which is the union of all the line segments from $O$ over $\partial B_i^{\pm}$. In fact
$$L_n=O\cone(P_n\cap(n-2){\rm -skeleton\,\,of}\,\,Q^n).$$
Since $F_j\subset\partial A$ on $\partial B_n^\pm$, we have
\begin{eqnarray*}
  \bigcup_{j=1}^{n-1} F_j\subset Q^n\cap\bigcup_{j=1}^{n-1}\{x_j=\mp1,\,x_n=\pm1,\,{x}_1+\cdots+\widehat{x}_j+\cdots+x_{n-1}=0\},
\end{eqnarray*}
and hence
\begin{eqnarray*}
(O\cone\partial B_n^+)\cup(O\cone\partial B_n^-)&\supset& O\cone\bigcup_{j=1}^{n-1} F_j\\
&\subset&\left[\bigcup_{j=1}^{n-1}\{x_j+x_n=0\}\right]\cap\,\{x_1+\cdots+x_n=0\}\cap Q^n.
\end{eqnarray*}
Therefore for $n=3,4$, $\hat{F}_j=\emptyset$ and we have
$$L_n=\left[\bigcup_{1\leq i\neq j\leq n}\{x_i+x_j=0\}\right]\,\cap P_n.$$
Actually, $L_3=L$ is the three straight lines on $P_3=H_0$ and $L_4$ is the three mutually orthogonal 2-planes on $P_4$:
 $$L_4=(\{x_1+x_2=0\}\cup\{x_1+x_3=0\}\cup\{x_1+x_4=0\})\cap P_4.$$
For $n\geq5$, however, because of the nonempty set $\cup_j \hat{F}_j$ we can just say that
\begin{eqnarray*}
\mathcal{H}^{n-2}\left(L_n\cap\left[\bigcup_{1\leq i\neq j\leq n}\{x_i+x_j=0\}\right]\,\cap P_n\right)>0,
\end{eqnarray*}
where $\mathcal{H}^{n-2}$ denotes the $(n-2)$-dimensional Hausdorff measure.
\begin{center}
\includegraphics[width=2.2in]{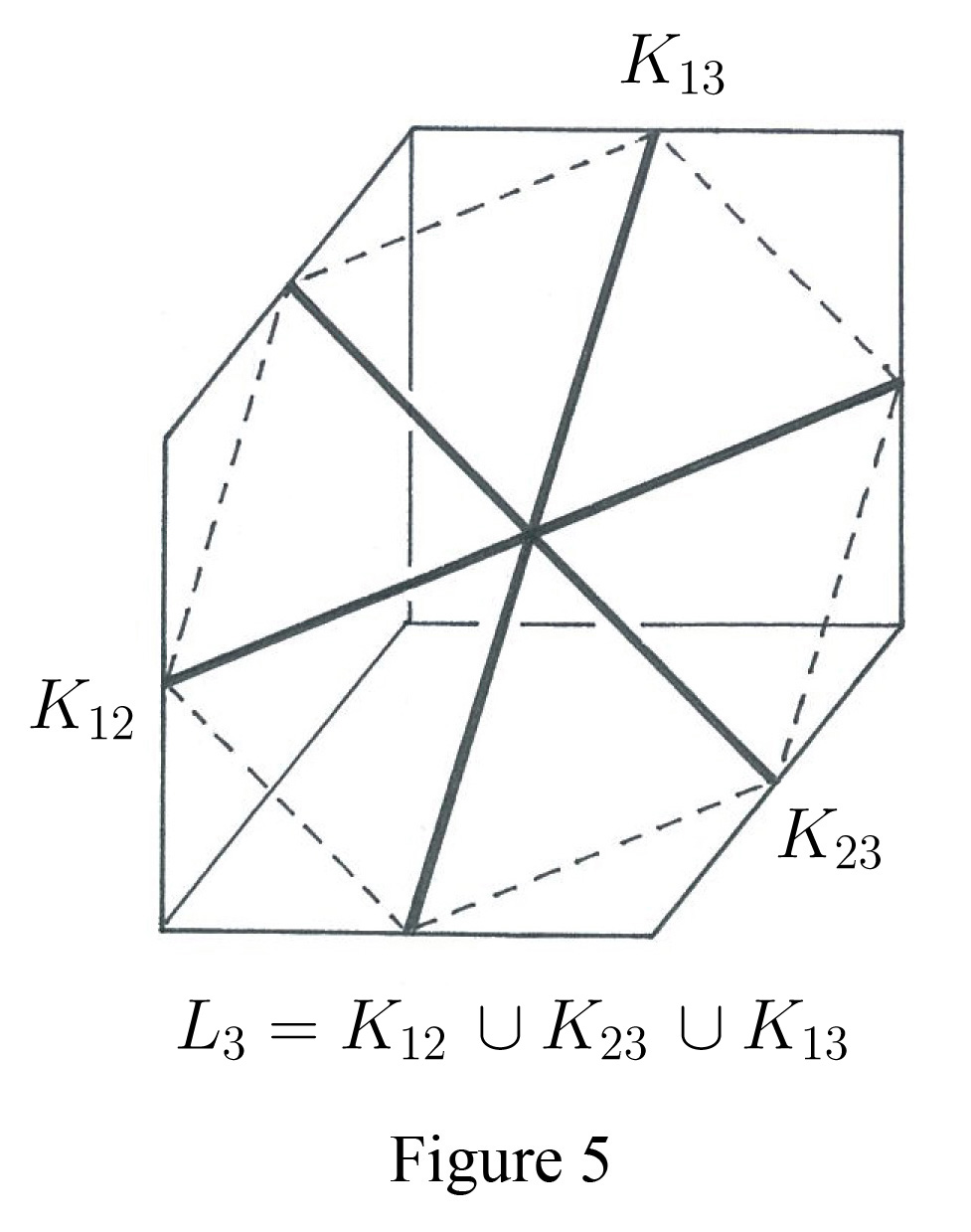}\\
\end{center}

Now fixing the spine $L_n$ of $P_n$, we want to perturb $P_n$ into a minimal hypersurface $\Sigma_4$ in $Q^n$. First, for an $(n-2)$-plane $K$ in $\mathbb R^n$, we need to define the $180^\circ$-rotation $\rho_K$ of $\mathbb R^n$ around $K$. Let $$K_{12}=\{x_1+x_2=0\}\cap\{x_1+\cdots+x_n=0\}.$$ Then both $u:=(1,1,0,\ldots,0)$ and $v:=(0,0,1,\ldots,1)$ are orthogonal to $K_{12}$. Hence the foot of perpendicular from $(x_1,\ldots,x_n)$ to $K_{12}$ is
$$(x_1,\ldots,x_n)  -\frac{x_1+x_2}{2}\,u-\frac{x_3+\cdots+x_n}{n-2}\,v.$$
Since the foot of perpendicular is the midpoint of ${\bf x}:=(x_1,\ldots,x_n)$ and $\rho_{K_{12}}({\bf x})$, we have
\begin{equation}\label{12}
\rho_{K_{12}}({\bf x})=\left(-x_2,-x_1,x_3-\frac{2}{n-2}(x_3+
\cdots+x_n),\ldots,x_n-\frac{2}{n-2}(x_3+\cdots+x_n)\right).
\end{equation}
In general, if we define
$$K_{ij}=\{x_i+x_j=0\}\cap\{x_1+\cdots+x_n=0\},$$
the $i$th and $j$th components of $\rho_{K_{ij}}(x_1,\ldots,x_n)$ are $-x_j$ and $-x_i$,
respectively.

Note that for all $n$,
$$\rho_{K_{ij}}(\tilde{P}_n)=\tilde{P}_n,\,\,\,{\rm if}\,\,\tilde{P}_n:=\{x_1+\cdots+x_n=0\}.$$
For $n=3,4$, we see that
\begin{equation}\label{34}
\rho_{K_{ij}}(L_n)=L_n,\,\,\,\,\rho_{K_{ij}}(Q^n)=Q^n\,\,\,\,{\rm and}\,\,\,\,\rho_{K_{ij}}(P_n)=P_n
\end{equation}
because
\begin{equation*}\label{n=3}
\rho_{K_{12}}(x_1,x_2,x_3)=(-x_2,-x_1,-x_3),
\end{equation*}
and
\begin{equation*}\label{n=4}
\rho_{K_{12}}(x_1,x_2,x_3,x_4)=(-x_2,-x_1,-x_4,-x_3).
\end{equation*}
Unfortunately, however, for $n\geq5$ we have
\begin{equation}\label{5}
\rho_{K_{ij}}(L_n)\neq L_n,\,\,\,\,\rho_{K_{ij}}(Q^n)\neq Q^n\,\,\,\,{\rm and}\,\,\,\,\rho_{K_{ij}}(P_n)\neq P_n,
\end{equation}
because
$$\rho_{K_{ij}}(\{x_k=1\})\neq\{x_l=-1\}\,\,\,{\rm for\,\,any}\,\,l\,\,{\rm if}\,\,k\neq i,\, j,$$
even though
$$\rho_{K_{ij}}(\{x_i=1\})=\{x_j=-1\}.$$

Let $\hat{O}=(2,0,\ldots,0)$ and consider $(O\cone B_1^+)\cup(\hat{O}\cone B_1^+)$ and $$\Gamma_1:=(O\cone\partial B_1^+)\cup(\hat{ O}\cone\partial B_1^+).$$ Here we want to deform
$(O\cone B_1^+)\cup(\hat{O}\cone B_1^+)$ into a minimal hypersurface spanning $\Gamma_1$. Let $\Pi_1$ be the orthogonal projection of $\mathbb R^n$ onto the hyperplane $\{x_2+\cdots+x_n=0\}$. Note that $\Pi_1(O\cone B_1^+\cup\hat{O}\cone B_1^+)$ contains $\overline{O\hat{O}}$. This fact, together with the convexity of $B_1^+$ in $\{x_1=1\}$, implies that $\Pi_1(O\cone B_1^+\cup\hat{O}\cone B_1^+)$ is convex on $\{x_2+\cdots+x_n=0\}$. Since $\Gamma_1$ is the graph of a piecewise linear function on $\Pi_1(\Gamma_1)$,
Jenkins-Serrin's theorem \cite{JS} states that
there exists a unique minimal hypersurface $\Sigma_0$ spanning $\Gamma_1$ as a graph over $\Pi_1(O\cone B_1^+\cup \hat{O}\cone B_1^+)$. (See Figure 6). Let $\Sigma_1=\Sigma_0\cap Q^n$. From the symmetry of $\Gamma_1$ with respect to $\{x_1=1\}$ it follows that $\Sigma_1$ is also symmetric with respect to $\{x_1=1\}$ and hence $\Sigma_1$ is perpendicular to $\{x_1=1\}$ along its boundary on $\{x_1=1\}$.
\begin{center}
\includegraphics[width=1.5in]{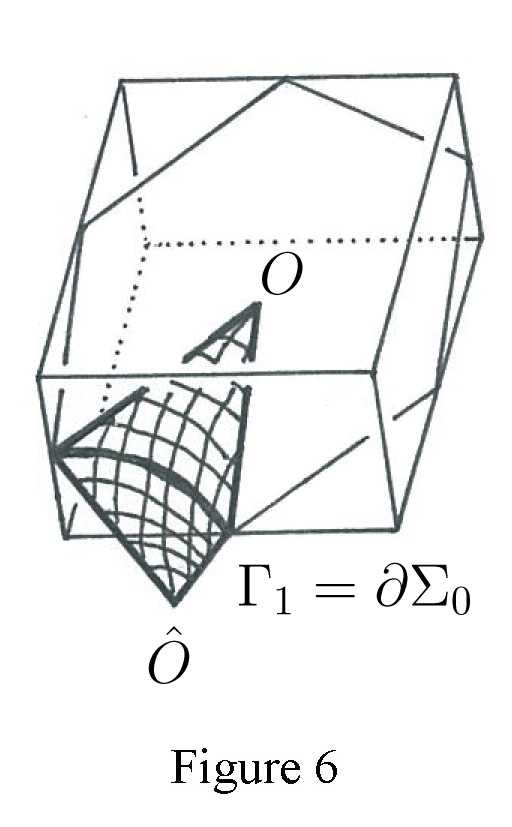}\\
\end{center}

Recall that $G_1$ is the subgroup of $O(n)$ consisting of all the isometries acting on $\{x_1,\ldots,x_n\}$ as permutations. Let $G_0$ be the subgroup of $G_1$ consisting of all the permutations of $\{x_1,\ldots,x_n\}$ fixing $x_1$. Note that $\Gamma_1$ is invariant under any $\psi\in G_0$. Hence the uniqueness of the minimal graph $\Sigma_0$ spanning $\Gamma_1$ implies that $\Sigma_0$ is also invariant under $G_0$.

We now try to extend $\Sigma_1$ analytically to obtain a complete minimal hypersurface in $\mathbb R^n$ as follows. Define $$\Sigma_2=\bigcup_{\psi\in G_1}\psi(\Sigma_1),\,\,\,\,\,\Sigma_3=\bigcup_{\psi\in G_1}\psi(\varphi(\Sigma_1)),\,\,\,\,\,\Sigma_4=\Sigma_2\cup\Sigma_3,$$
where $\varphi:\mathbb R^n\rightarrow\mathbb R^n$, $\varphi(x)=-x$. (See Figure 7.) Clearly $\psi(L_n)=L_n$ for any $\psi\in G_2$. From the invariance of $\Sigma_0$ under $G_0$ we see that if $\psi_1(\Sigma_1)$ and $\psi_2(\Sigma_1),\psi_1,\psi_2\in G_1,$ span the same boundary inside $Q^n$, i.e., if $\psi_1(\Sigma_1)\setminus\partial Q^n=\psi_2(\Sigma_1)\setminus\partial Q^n$, then they must coincide. Hence both $\Sigma_2$ and $\Sigma_3$ are embedded. Moreover, we have
$$\partial\Sigma_2\cap\partial\Sigma_3 =L_n,\,\,\,\,\,\partial\Sigma_2\setminus L_n\subset\partial Q^n,\,\,\,\,\,\partial\Sigma_3\setminus L_n\subset\partial Q^n.$$
Hence $\Sigma_4$ is a connected, $C^0$, piecewise analytic manifold with $\partial\Sigma_4\subset\partial Q^n$.
\begin{center}
\includegraphics[width=2.5in]{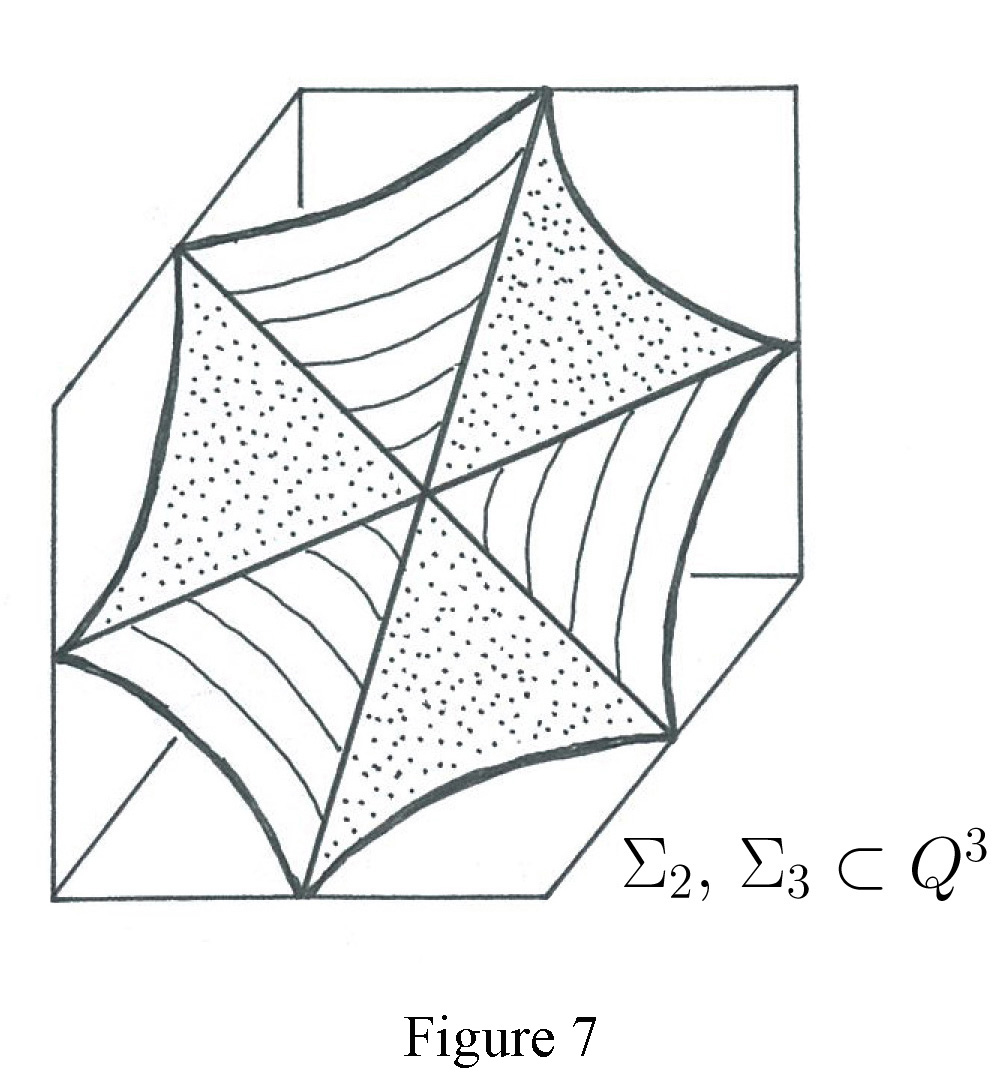}\\
\end{center}

We claim here that $\Sigma_4$ is an analytic extension of $\Sigma_1$ only when $n=3,4$. From the well-known removable singularity theorem (Theorem 1.4, \cite{HL}) it follows that the following four statements are equivalent:
\begin{quote}
$\,\,\,\,\,\,\,\,\Sigma_4$ is an analytic extension of $\Sigma_1$.\\
$\Leftrightarrow$ The tangent planes to $\Sigma_1$ and to $\Sigma_4\setminus\Sigma_1$ coincide at every
 point of $\Sigma_1\cap\Sigma_4\cap K_{12}$.\\
$\Leftrightarrow$ $\rho_{K_{12}}(\Sigma_1)$ is a subset of $\Sigma_4$.\\
$\Leftrightarrow$
\begin{equation}\label{coincide}
\rho_{K_{12}}(O\cone \partial B_1^+)=O\cone\partial B_2^-.
\end{equation}
\end{quote}
Remark that
$$ (O\cone\cup_jF_j)\cap(O\cone\partial B_1^+)\,\,\subset\,\,\cup_{i\neq1}\{x_1+x_i=0\}
\cap\,\{x_1+\cdots+x_n=0\}\cap Q^n$$
and
\begin{equation}\label{21}
(O\cone\cup_jF_j)\cap(O\cone\partial B_2^-)\,\,\subset\,\,\cup_{i\neq2}\{x_2+x_i=0\}
\cap\,\{x_1+\cdots+x_n=0\}\cap Q^n.
\end{equation}
From \eqref{12} we see that the sum of the second and $k$-th components of $\rho_{K_{12}}(x_1,\ldots,x_n)$ for $(x_1,\ldots,x_n)\in (O\cone \cup_j F_j)\cap (O\cone\partial B_1^+)$ equals
$$-x_1+x_k-\frac{2}{n-2}(x_3+\cdots+x_n),$$
which does not vanish when $x_1+x_i=0$, $i\neq1$, and $x_1+\cdots+x_n=0$, if $k\geq3$ and $n\geq5$. It follows from \eqref{21} that $O\cone\partial B^+_1$ cannot be mapped by $\rho_{K_{12}}$ to $O\cone\partial B^-_2$ if $n\geq5$, which contradicts \eqref{coincide}. Therefore $\Sigma_4$ cannot be an analytic extension of $\Sigma_1$ if $n\geq5.$ If $n=3,4$, however, \eqref{coincide} follows from \eqref{34}
and therefore $\Sigma_4$ is an embedded analytic extension of $\Sigma_1$, as claimed.

From here on, assume $n=4$. Note that $\Sigma_4$ meets $\partial Q^4$ orthogonally. Therefore repeated reflections of $\mathbb R^n$ across the hyperplanes $\{x_i=2k+1\}$ for all $i=1,2,3,4$ and for all integers $k$ give rise to the desired complete embedded analytic  minimal hypersurface $\Sigma_P$ in $\mathbb R^4$. Obviously $\Sigma_P$ is periodic in each direction of the four coordinate axes of $\mathbb R^4$.

Interestingly, $\Sigma_4$ can be interpreted as an equator in $Q^4$ between the two {\it poles} $p^+=(1,1,1,1)$ and $p^-=(-1,-1,-1,-1)$ of $\partial Q^4$. Define two 4-prong {\it polar grids} $\gamma^+=\cup_{i=1}^4\ell_i^+$ containing $p^+$ and $\gamma^-=\cup_{i=1}^4\ell^-_i$ containing $p^-$ by
$$\ell_1^+=\{(x_1,1,1,1):-1\leq x_1\leq1\},\ldots,\ell^+_4=\{(1,1,1,x_4):-1\leq x_4\leq1\},$$
$$\ell_1^-=\{(x_1,-1,-1,-1):-1\leq x_1\leq1\},\ldots,\ell^-_4=\{(-1,-1,-1,x_4):-1\leq x_4\leq1\}.$$
Let $\gamma_\varepsilon^+$ be an $\varepsilon$-tubular neighborhood of $\gamma^+$ in $Q^4$ and $\gamma_\varepsilon^-$ that of $\gamma^-$ in $Q^4$. Then the following four sets are diffeomorphic:
\begin{equation}\label{diffeo}
\partial\gamma_\varepsilon^+\setminus\partial Q^4\,\,\approx\,\,P_4\,\,\approx\,\,\Sigma_4\,\,\approx\,\,\partial\gamma_\varepsilon^-\setminus\partial Q^4.
\end{equation}
It is in this sense that $\Sigma_4$ is called an equator between the two poles.

Define the {\it 1-dimensional grid} $\gamma_{\infty}^+$ ($\gamma_{\infty}^-$, respectively) in $\mathbb R^4$ to be the set of all lines parallel to the four coordinate axes, consisting of all the points $(x_1,x_2,x_3,x_4)$ three components of which are integers $\equiv1\,({\rm mod}\, 4)$ ($\equiv-1\,({\rm mod}\,4)$, respectively). Then $\Sigma_P$ can be viewed ``roughly" as an equi-distance set of the grids $\gamma_\infty^+$ and $\gamma_\infty^-$ in the following sense. Let
$$(2Q)^4=[-3,1]^4=\{(x_1,x_2,x_3,x_4):-3\leq x_i\leq1\}.$$
Identifying the two points on the parallel faces of $(2Q)^4$, one can make $(2Q)^4$ into a four-dimensional torus $T^4$. With this identification $\Sigma_P\cap (2Q)^4$ becomes a compact 3-dimensional embedded minimal hypersurface $\Sigma_P'$ in $T^4$. If follows from \eqref{diffeo} that $\Sigma_P'$ is diffeomorphic to the boundary of a tubular neighborhood of $\gamma_{\infty}^+$ in $T^4$, and to that of $\gamma_{\infty}^-$ in $T^4$ as well. One can foliate  $T^{4}\setminus(\gamma_{\infty}^+\cup \gamma_{\infty}^-)$ by a 1-parameter family of $3$-dimensional hypersurfaces which are diffeomorphic to the boundary of a tubular neighborhood of $\gamma_{\infty}^+$  and which sweep out $T^4$ from $\gamma_{\infty}^+$ to $\gamma_{\infty}^-$. Applying the minimax argument, one can find a compact embedded minimal hypersurface $\Sigma_T$ from this family of hypersurfaces. $\Sigma_T$ should be the same as $\Sigma_P'$. And one easily sees that $\pi_1(\Sigma_P')$ is the free group with $4$ generators. The hypersurface $\Sigma_P$ divides $\mathbb R^4$ into two congruent labyrinths as one of them is mapped to the other by $\rho_{K_{12}}$.

In conclusion, we summarize the properties of $\Sigma_P$ as follows.

\begin{theorem}
There exists a minimal hypersurface $\Sigma_P$ in $\mathbb R^4$ which generalizes the Schwarz $P$-surface of $\mathbb R^3$ with the following properties{\rm :}

{\rm a)} $\Sigma_P$ is embedded and periodic in each direction of the four coordinate axes of $\mathbb R^4$.

{\rm b)} $\Sigma_P$ divides $\mathbb R^4$ into two congruent labyrinths.

{\rm c)} One can normalize the coordinates of $\mathbb R^4$ such that $\Sigma_P$ has period $4$ in each direction. Moreover, for every point $p\in\mathbb R^4$ with coordinates $(2k,2l,2m,2n)$, $k,l,m,n$$:$  integers, three mutually orthogonal planes pass through $p$ and totally lie in $\Sigma_P$.

{\rm d)} Let $T^4$ be the $4$-dimensional torus obtained by identifying the parallel faces of the cube $[-3,1]^4$ in $\mathbb R^4$. Then $\Sigma_P\cap[-3,1]^4$ becomes a compact embedded minimal hypersurface $\Sigma_P'$ in $T^4$. Let $\gamma^4\subset\mathbb R^2\subset\mathbb R^4$ be a four-leaved rose, i.e., the union of four Jordan curves which intersect each other only at one given point. Then $\Sigma_P'$ is diffeomorphic to the boundary of a tubular neighborhood of $\gamma^4$ in $\mathbb R^4$ and $\pi_1(\Sigma_P')$ is the free group with $4$ generators.
\end{theorem}

{\bf Remark 1.} In conclusion, Schwarz's minimal surface has been constructed in $\mathbb R^3$ and $\mathbb R^4$ but not in $\mathbb R^n$ for $n\geq5$. Strangely, this situation is similar to a famous classical problem in algebra: solvability of the cubic and quartic equations in radicals and insolvability of the quintic. This may not be a pure coincidence, remarking that permutations of $\{x_1,\cdots,x_n\}$ are critically used in the construction of $\Sigma_4$ and that Galois theory is based on the group of permutations. Moreover, as the roots of an algebraic equation are required to be expressed only with the radicals, we have strongly required that the spine $L_n$ be totally geodesic.\\

{\bf Remark 2.} Our fundamental piece $\Sigma_1$ can be analytically extended to a complete embedded minimal hypersurface in $\mathbb R^n$ for $n=3,4$, but not for $n\geq5$. However, our guess is that such a complete embedded minimal hypersurface $\Sigma_P$ should exist even in $\mathbb R^n$ for $n\geq5$. Near $\Sigma_1$ there should exist an analytic minimal hypersurface $\Sigma_1'$ whose boundary is more flexible than totally geodesic $\Sigma_1\cap\Gamma_1$ and which is orthogonal to $\partial Q^n$ so that $\Sigma_1'$ may extend to a complete embedded minimal hypersurface $\Sigma_P$ in $\mathbb R^n$. $\Sigma_P$ should be a minimax solution in a 1-parameter family of hypersurfaces sweeping out $\mathbb R^n$ from $\gamma_{\infty}^+$ to $\gamma_{\infty}^-$. Here $\gamma_{\infty}^+$ and $\gamma_{\infty}^-$ are the dual pair of all lines consisting of the points $(x_1,\ldots,x_n)$, $(n-1)$-components of which are integers $\equiv1\,({\rm mod}\,4)$ and $\equiv-1\,({\rm mod}\,4)$, respectively.

\section{Schwarz's $D$-surface}
Schwarz's $D$-surface $R$ is one of the simplest among dozens of triply periodic minimal surfaces in $\mathbb R^3$. Its fundamental piece $R_0$ spans the skew quadrilateral with vertex angles $\pi/3,\pi/2,\pi/2,\pi/2$. It is interesting to notice that $R_0$ is a quarter of Schwarz's initial surface $S_0$ (Figure 8). Therefore Schwarz's $P$-surface and $D$-surface are the conjugate minimal surfaces.
\begin{center}
\includegraphics[width=1.9in]{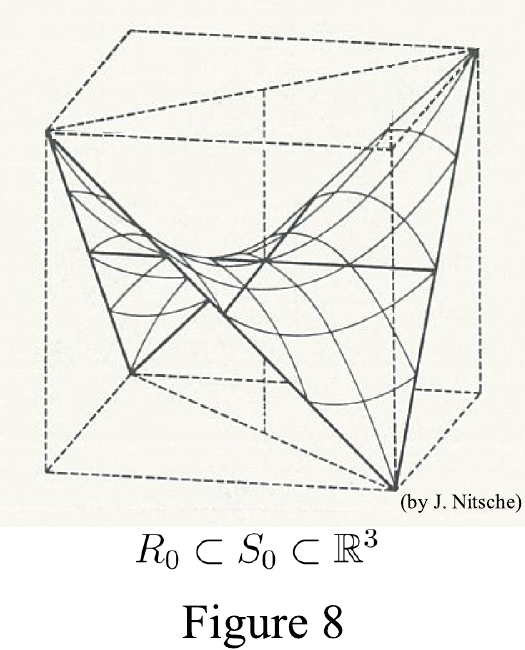}\\
\end{center}

Thanks to the single vertex angle of $\pi/3$ in $R_0$, six congruent pieces surrounding that vertex constitute a hexagonal minimal surface $R_1$ whose vertex angles are all $\pi/2$. (See Figure 9.) \begin{center}
\includegraphics[width=4.7
in]{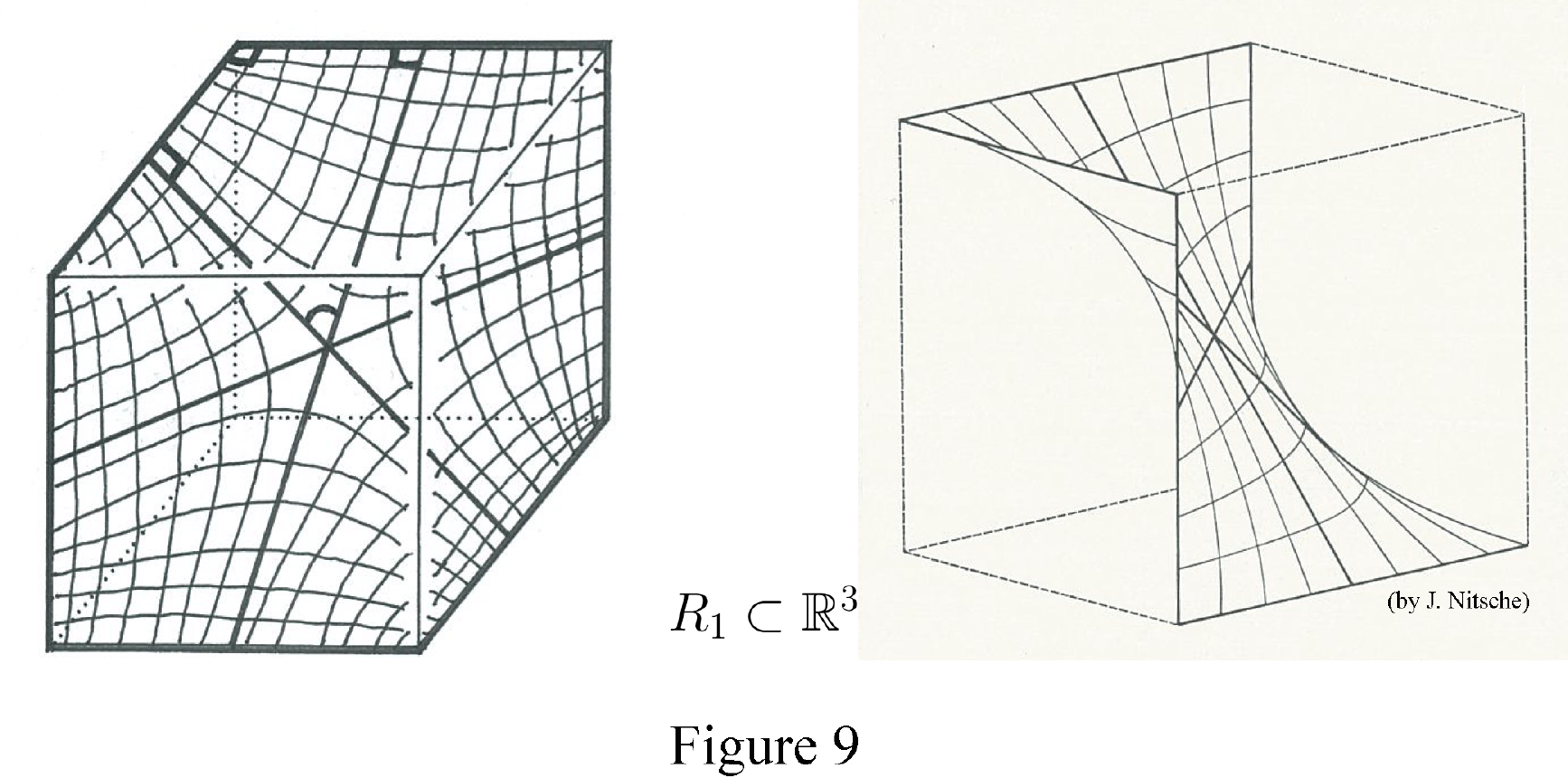}\\
\end{center}Since $\partial R_1$ is a subset of the 1-skeleton of a cube, $R_1$ can be extended to the complete embedded minimal surface $R$. We can generalize this nice property of $R_1$ in higher dimension to construct the higher-dimensional Schwarz $D$-surface in $\mathbb R^n$ for any $n$ as follows.

\begin{theorem}
There exists an $(n-1)$-dimensional Schwarz's $D$-surface $\Sigma_D$ in $\mathbb R^n$ for any $n\geq4${\rm :}

{\rm a)} $\Sigma_D$ is complete and embedded.

{\rm b)} $\Sigma_D$ is periodic in every direction of the coordinate axes of $\mathbb R^n$.

{\rm c)} If $\Sigma_D$ is normalized to have period $2$ in each coordinate direction, at every point $p\in\mathbb R^n$ with odd integer coordinates $\Sigma_D$ completely contains $n-1$ $(n-2)$-planes.
\end{theorem}

\begin{proof}
In the preceding section $\Sigma_4$ is interpreted as an equator in $Q^4$ between the two poles $(1,1,1,1)$ and $(-1,-1,-1,-1)$. Here we introduce another type of equator in $\tilde{Q}^n:=[0,1]^n$ between the poles $p^0=(0,\ldots,0)$ and $p^1=(1,\ldots,1)$ in $\tilde{Q}^n$. $\tilde{Q}^n$ has $2n$ faces $F_i^0:=\{x_i=0\}\cap\partial \tilde{Q}^n$ and $F_i^1:=\{x_i=1\}\cap\partial \tilde{Q}^n$ for $ i=1,\ldots, n$. Define
$$F^0=\bigcup_{i=1}^nF_i^0,\,\,\,\,F^1=\bigcup_{i=1}^nF_i^1,\,\,\,\,\Gamma_2=F^0\cap F^1.$$
Clearly
$$\Gamma_2=\partial F^0=\partial F^1.$$
$\Gamma_2$ is homeomorphic to $\mathbb S^{n-2}$. Among $2^n$ vertices of $\tilde{Q}^n$, $\Gamma_2$ contains $2^n-2$ of them, leaving out only $p^0$ and $p^1$. As a CW-complex $\tilde{Q}^n$ has the $(n-2)$-skeleton which consists of $(n-2)$-dimensional cubes. The total number of $(n-2)$-dimensional cubes in the $(n-2)$-skeleton of $\tilde{Q}^n$ is $2n(n-1)$. Half of them contains either $p^0$ or $p^1$. Hence $\Gamma_2$ contains $n(n-1)$ cubes.

Let $\Pi_2$ be the orthogonal projection of $\mathbb R^n$ onto the hyperplane $\tilde{P}_n=\{x_1+\cdots+x_n=0\}$. Then $\Pi_2(\Gamma_2)$ bounds a convex region $U:=\Pi_2(\tilde{Q}^n)\subset \tilde{P}_n$. The vertices of $U$ are the projections under $\Pi_2$ of all the vertices of $\tilde{Q}^n$ except for $p^0$ and $p^1$.  Since $\Gamma_2$ is the graph of a piecewise linear function defined on $\Pi_2(\Gamma_2)$, Jenkins-Serrin's theorem gives a unique minimal hypersurface $\Sigma_5$ spanning $\Gamma_2
$ as a graph over $U$. $\Sigma_5=R_1$ in case $n=3$. Obviously,
$$\Sigma_5\subset\tilde{Q}^n\,\,\,{\rm because}\,\,\,\Gamma_2\subset\partial\tilde{Q}^n.$$
As $\Gamma_2$ is invariant under the isometries of $\mathbb R^n$ acting on $\{x_1,\ldots,x_n\}$ as permutations, so is $\Sigma_5$. Hence one can see that any pair of antipodal vertices $\{p,q\}$ (i.e., ${\rm dist}(p,q)=\sqrt{n}$) of $\tilde{Q}^n$ uniquely determines a minimal equator between them which is congruent to $\Sigma_5$. Let's denote this minimal equator by $\Sigma_{\{p,q\}}$.

Define
$$2\tilde{Q}^n=[-1,1]\times\cdots\times[-1,1]\subset\mathbb R^n.$$
The hyperplanes $\{x_i=0\},i=1,\ldots,n$, split $2\tilde{Q}^n$ into $2^n$ subcubes each of which is congruent to $\tilde{Q}^n$. One can make $2\tilde{Q}^n$ into an $n$-dimensional checkerboard by selecting the congruent subcubes in an alternating way. Let's denote the ``black" part of the checkerboard containing $\tilde{Q}^n$ by $\frac{2\tilde{Q}^n}{2}$. In each subcube of  $\frac{2\tilde{Q}^n}{2}$ we want to put a minimal hypersurface congruent to $\Sigma_5$ as follows. Let $Q$ be a copy of $\tilde{Q}^n$ in $\frac{2\tilde{Q}^n}{2}$. $Q$ has a unique vertex $p_Q$ which is antipodal to $O$ and then $Q$ has a unique minimal equator $\Sigma_{\{O,\,p_Q\}}$ determined by the antipodal pair $\{O,p_Q\}$.

Combining all the minimal hypersurfaces $\Sigma_{\{O,\,p_Q\}}$ in each subcube $Q$ of $\frac{2\tilde{Q}^n}{2}$, we define
$$\Sigma_6=\bigcup_{Q\subset\frac{2\tilde{Q}^n}{2}}\Sigma_{\{O,\,p_Q\}}.$$
\begin{center}
\includegraphics[width=2.2in]{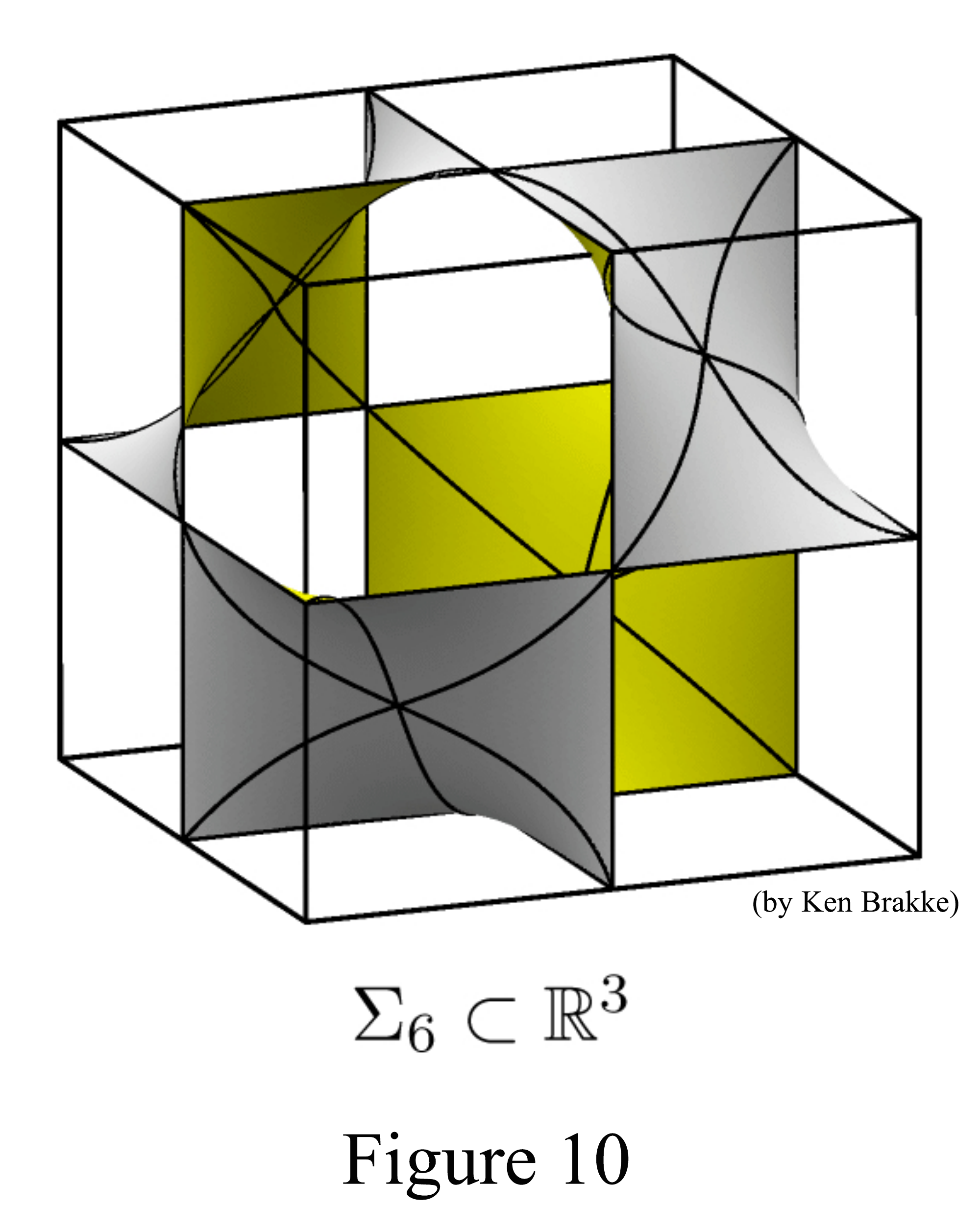}\\
\end{center}
Since $\Gamma_2=\partial F^0$ and $F^0\subset\cup_i\,\{x_i=0\}$, $\partial\Sigma_6$ is a subset of $\cup_i\,\{x_i=0\}$. And since $\Gamma_2=\partial F^1$ and $F^1\subset\partial(2\tilde{Q}^n)$, $\partial\Sigma_6$ lies on the boundary of $2\tilde{Q}^n$.
Therefore
\begin{equation}\label{line}
\partial\Sigma_6=\partial(2\tilde{Q}^n)\cap\bigcup_{i=1}^n\,\{x_i=0\}.
\end{equation}
Let $q_i^+,q_i^-$ be the points on the $x_i$-axis whose $x_i$-coordinates equal $1,-1$, respectively. Then $q_i^+,q_i^-\in\Gamma_2$ and $\Gamma_2\subset F_i^1=\{x_i=1\}\cap\partial\tilde{Q}^n$ in a neighborhood of $q_i^+$ for all $i=1,\ldots,n$. Hence $\Sigma_5$ is tangent to the face $F_i^1$ of $\tilde{Q}^n$ at $q_i^+$. It follows that $\Sigma_6$ is also tangent to the faces of $2\tilde{Q}^n$ at $q_1^+,\ldots,q_n^+$ and at $q_1^-,\ldots,q_n^-$.

In order to extend $\Sigma_6$ into a complete minimal hypersurface we need to understand the behavior of $\Sigma_6$ near the point $q_1^+=(1,0,\ldots,0)\in\Gamma_2$. In a neighborhood of $q_1^+$ $\Sigma_5$ is a graph over $V:=\{(1,x_2,\ldots,x_n):x_i\geq0, i=2,\ldots,n\}\subset\{x_1=1\}$. $\partial\Sigma_5$ contains all the $(n-2)$-planes $\{x_1=1\}\cap\{x_i=0\}\cap\partial V$ in a neighborhood of $q_1^+$. Hence by the $180^\circ$-rotations of $\Sigma_5$ around all these $(n-2)$-planes $\Sigma_5$ can be analytically extended to a minimal hypersurface $\Sigma_7$ which is a graph over $\{x_1=1\}$ in the same neighborhood. For $i=2,\ldots,n$, let $\rho_i$ be the rotation of $\mathbb R^n$ about the $(n-2)$-plane $\{x_1=1\}\cap\{x_i=0\}$ and let $\lambda_i$ be the reflection in $\mathbb R^n$ across the $(n-1)$-plane $\{x_i=0\}$. Since
$$\rho_2(x_1,\ldots,x_n)=(2-x_1,-x_2,x_3,\ldots,x_n),$$
one gets
$$\rho_i\circ\rho_j=\lambda_i\circ\lambda_j$$
and hence
$$\rho_i\circ\rho_j\left(\frac{2\tilde{Q}^n}{2}\right)=\frac{2\tilde{Q}^n}{2},\,\,\,\rho_i\circ\rho_j(O)=
O,\,\,\,\rho_i\circ\rho_j(\Sigma_6)=\Sigma_6.$$ It follows that
\begin{equation}\label{67}
\Sigma_6=\Sigma_7\,\,\,{\rm in}\,\,\,{\rm a}\,\,\,{\rm neighborhood}\,\,\,{\rm of}\,\,\,q_1^+.
\end{equation}

Remember that each subcube $Q$ of $\frac{2\tilde{Q}^n}{2}$ in the checkerboard has a unique vertex $p_Q$ antipodal to $O$ and contains a unique minimal equator $\Sigma_{\{O,\,p_Q\}}$. Let $\mathcal{L}=\bigcup_{Q\subset\frac{2\tilde{Q}^n}{2}}\{p_Q\}$. $\mathcal{L}$ forms an alternating subset in the set of $2^n$ vertices of $2\tilde{Q}^n$. Clearly $\mathcal{L}$ consists of $2^{n-1}$ vertices and completely determines $\Sigma_6$ in the sense that $\Sigma_6=\cup_{q\in\mathcal{L}}\Sigma_{\{O,\,q\}}$. Consider $\mathcal{L}\cap\{x_1=-1\}$ which consists of $2^{n-2}$ vertices of $2\tilde{Q}$. Choose any $q\in\mathcal{L}\cap\{x_1=-1\}$ and let $\tau$ be the parallel translation of $\mathbb R^n$ by $2$ in the direction of $x_1$-axis. Then
$$\tau(O)=(2,0,\ldots,0),\,\,\,\tau(q)\in\{x_1=1\},\,\,\,\tau(q)\notin\mathcal{L}.$$
However, there exists $\bar{q}\in\mathcal{L}$ such that $\tau(q)=\rho_i(\bar{q})$ for some $i=2,\ldots,n$. Moreover,
$\tau(O)=\rho_i(O).$
Therefore we have
$$\tau(\Sigma_{\{O,\,q\}})=\rho_i(\Sigma_{\{O,\,\bar{q}\}}).$$
Since $\Sigma_{\{O,\,\bar{q}\}}\subset \Sigma_6$ and $\rho_i(\Sigma_6)=\Sigma_7$ in a neighborhood of $q_1^+$, it follows that
\begin{equation}\label{1267}
\tau(\Sigma_6)=\Sigma_7\,\,\,{\rm in}\,\,\,{\rm a}\,\,\,{\rm neighborhood}\,\,\,{\rm of}\,\,\,q_1^+.
\end{equation}

Viewing $\Sigma_7$ as a minimal graph over $\{x_1=1\}$ in a neighborhood of $q_1$,  we see that the sign of $\Sigma_7$ is alternating on the components of $\{x_1=1\}\setminus\cup_{i=2}^n\{x_i=0\}$. In a neighborhood of $q_1^+$ $\Sigma_6$ constitutes the part where $\Sigma_7$ is negative and $\tau(\Sigma_6)$ positive.

The point $q_1^+$ is the center of the face $\{x_1=1\}$ of $2\tilde{Q}^n$. Again by the invariance of $\Sigma_5$ under the permutations of $\{x_1,\ldots,x_n\}$ the property of $\Sigma_6$ around $q_1^+$ as in \eqref{67} and \eqref{1267} should also hold around the center $q_i^+$ of every face $\{x_i=1\}$ of $2\tilde{Q}^n$. Hence we can extend $\Sigma_6$ into the complete embedded minimal hypersurface $\Sigma_D$ by periodically translating $\Sigma_6$ (with period of 2) in every direction of the coordinate axes of $\mathbb R^n$:
$$\Sigma_D=\bigcup_{k_1,\ldots,k_n:\,{\rm integers}}\tau_{2k_1,\ldots,2k_n}(\Sigma_6),$$
where $\tau_{2k_1,\ldots,2k_n}:\mathbb R^n\rightarrow\mathbb R^n$ is the parallel translation defined by
$$\tau_{2k_1,\ldots,2k_n}(x_1,\ldots,x_n)=(x_1+2k_1,\ldots,x_n+2k_n).$$

By the removable singularity theorem \cite{HL} $\Sigma_D$ is analytic everywhere. Finally \eqref{line} implies that $\Sigma_D$ contains $n-1$ $(n-2)$-planes at $q_i^\pm$ with odd integer coordinates.
\end{proof}

\section{Scherk's second surface}
Scherk's minimal surfaces were found in 1834. After the catenoid(1744) and helicoid(1774), they were the third example(s) of minimal surfaces. Scherk used the method of separation of variables to find the equations
$$z=\log\cos x - \log\cos y \,\,\,\, {\rm and} \,\,\,\, \sin z=\sinh x\cdot \sinh y$$
for the first surface and the second surface, respectively. Scherk's first surface is doubly periodic and the second surface singly periodic. It turns out that these two are conjugate minimal surfaces. The second surface is asymptotic to two orthogonal planes. In fact, H. Karcher \cite{K} has found that there exist minimal saddle towers in $\mathbb R^3$ which are asymptotic to $k$ planes intersecting each other along a line at equal angles of $\pi/k$ for any integer $k\geq2$. In this section we construct the higher dimensional generalizations of Scherk's second surface and the saddle towers. These hypersurfaces are asymptotic to $k$ hyperplanes meeting each other at $\pi/k$-angles for any integer $k\geq2$. The key idea of our method is to use the catenoid as a barrier in the Dirichlet problem. It should be mentioned that F. Pacard \cite{P} has constructed similar hypersurfaces for $k=2$ using a different method (desingularization procedure).

\begin{theorem}
For any integer $k\geq2$ there exists an embedded minimal hypersurface $\Sigma_S$ in $\mathbb R^n$ satisfying the following properties{\rm :}

{\rm a)} $\Sigma_S$ is asymptotic to $k$  hyperplanes $\Pi_1,\ldots,\Pi_k$ meeting each other along the $(n-2)$-plane $P^{n-2}:=\{x_1=0,\,x_n=0\}$ at equal angles of $\pi/k$.

{\rm b)} $\Sigma_S$ is periodic in $n-2$ pairwise orthogonal directions of $P^{n-2}$.

{\rm c)} Given any positive real numbers $a_2,\ldots,a_{n-1}$, consider the union $\mathcal{P}^{n-3}$ of all the $(n-3)$-planes $\{x_1=0,x_n=0,x_2=ma_2\},\{x_1=0,x_n=0,x_3=ma_3\},\ldots,\{x_1=0,x_n=0,x_{n-1}=ma_{n-1}\}$ in $P^{n-2}$ for every integer $m$. $\mathcal{P}^{n-3}$ divides $P^{n-2}$ into $(n-2)$-dimensional rectangular cubes which are all congruent to $(0,a_2)\times(0,a_3)\times\cdots\times(0,a_{n-1})$. Let $\ell_1,\ldots,\ell_k$ be the lines in the $x_1x_n$-plane which are contained in $\Pi_1,\ldots,\Pi_k$, respectively, so that $\Pi_i=P^{n-2}\times\ell_i$, $i=1,\ldots,k$. Then $\Sigma_S$ contains $\mathcal{P}^{n-3}\times\ell_i$ for all $i=1,\ldots,k$.
\end{theorem}

To prove this theorem we need to introduce the higher-dimensional catenoid $\mathcal{C}^{n-1}\subset\mathbb R^n$. $\mathcal{C}^{n-1}$ is obtained by rotating a generating curve $C:\,x_n=f(x_1)$ of the $x_1x_n$-plane through the $SO(n-1)$ action on the $x_2\cdots x_n$-plane. The resulting hypersurface has zero mean curvature if  and only if
$$x_nx_n''-(n-2)\{1+(x_n')^2\}=0.$$
It is interesting to note that $\mathcal{C}^{n-1}$ lies in a slab of $\mathbb R^n$ if $n\geq4$.
Since the minimality of a hypersurface is invariant under homothety, we can assume that $\mathcal{C}^{n-1}$ lies in the slab $\{-1< x_1<1\}$ and is asymptotic to the boundaries of the slab. Let's define the upper half catenoid
$$\frac{1}{2}\mathcal{C}^{n-1}=\mathcal{C}^{n-1}\cap\{x_n\geq0\}.$$
$\frac{1}{2}\mathcal{C}^{n-1}$ is the graph of a  nonnegative function $x_n=g(x_1,\ldots,x_{n-1})$. Since one can find $a>0$ such that $f(x_1)\geq a$ for $-1<x_1<1$ and $f(0)=a$, the domain of definition of $g$ contains the solid cylinder $D^{n-1}:=\{-1<x_1<1, \,x_2^2+\cdots+x_{n-1}^2< a^2,\,x_n=0\}$. We are going to use $\frac{1}{2}\mathcal{C}^{n-1}$ as a barrier in the proof of the theorem.

\begin{proof}
By the invariance of minimality of $\Sigma_S$ under homothety we may assume
\begin{equation}\label{ai}
a_i<\frac{a}{\sqrt{n-2}}\,\,\,\, {\rm for\, all}\,\,\,i=2,\ldots,n-1.
\end{equation}
Let $Q_b^{n-1}=[-b,b]\times[0,a_2]\times\cdots\times[0,a_{n-1}]$ be a closed cube in the horizontal hyperplane $\{x_n=0\}$. Then by \eqref{ai} we have $Q_b^{n-1}\cap\{-1<x_1<1\}\subset D^{n-1}$. For any integer $k\geq2$, define a function on the infinite cube $Q_\infty^{n-1}$
\begin{equation}\label{hk}
h_k(x_1,\ldots,x_{n-1})=c_k|x_1|,
\end{equation}
where $c_k>0$ is to be determined.

The graph of $x_n=h_k(x_1,\ldots,x_{n-1})$ over $Q_b^{n-1}$ is piecewise planar (V-shaped) with angle $\theta_k$ along the sharp edge over $\{0\}\times[0,a_2]\times\cdots\times[0,a_{n-1}]$. Determine $c_k$ in such a way that $\theta_k=\pi/k$. We want to replace ${\rm graph}(h_k)$ with a minimal hypersurface by finding a function $\tilde{h}_{k,b}$ on $Q_b^{n-1}$ such that $\tilde{h}_{k,b}=h_k$ on $\partial Q_b^{n-1}$ and the graph of $x_n=\tilde{h}_{k,b}(x_1,\ldots,x_{n-1})$ is  minimal. By Jenkins-Serrin \cite{JS} such a $\tilde{h}_{k,b}$  exists.
From \eqref{hk} we see that $h_{k}\leq c_k$ on $Q_1^{n-1}$. Hence
\begin{equation}\label{hkb}
\tilde{h}_{k,1}\leq c_k\,\,\,{\rm on}\,\,\,Q_1^{n-1}.
\end{equation}
 Clearly we have
$$\tilde{h}_{k,b_1}<\tilde{h}_{k,b_2}\,\,\,\,{\rm on}\,\,\,Q_{b_1}^{n-1}\,\,\,{\rm if}\,\,\,b_1<b_2.$$
When $b$  increases, we need to show that $\tilde{h}_{k,b}$ cannot become much bigger than $g$ on $Q_1^{n-1}$.
Suppose $\tilde{h}_{k,1}=g+c_k$ at a point $p_1$ of $Q_1^{n-1}$. Since $g>0$ on $Q_1^{n-1}$, from \eqref{hkb} we see that $p_1$ cannot be a boundary point of $Q_\infty^{n-1}$. $p_1$ cannot be a boundary point of the slab $\{-1<x_1<1\}$ either, because $g=\infty$ there. Hence $p_1$ must be an interior point of $Q_1^{n-1}$. Then there should exist an interior point  $p_2$ of $Q_1^{n-1}$ and $c>0$ such that $\tilde{h}_{k,1}\leq g+c_k+c$ on $Q_1^{n-1}$ and equality holds at $p_2$. But this contradicts the maximum principle. Hence
$$\tilde{h}_{k,b}<g+c_k\,\,\,\,{\rm on}\,\,\,Q_1^{n-1}\,\,\,{\rm for}\,\,\,{\rm any}\,\,\,b\geq1.$$

\begin{center}
\includegraphics[width=2.7in]{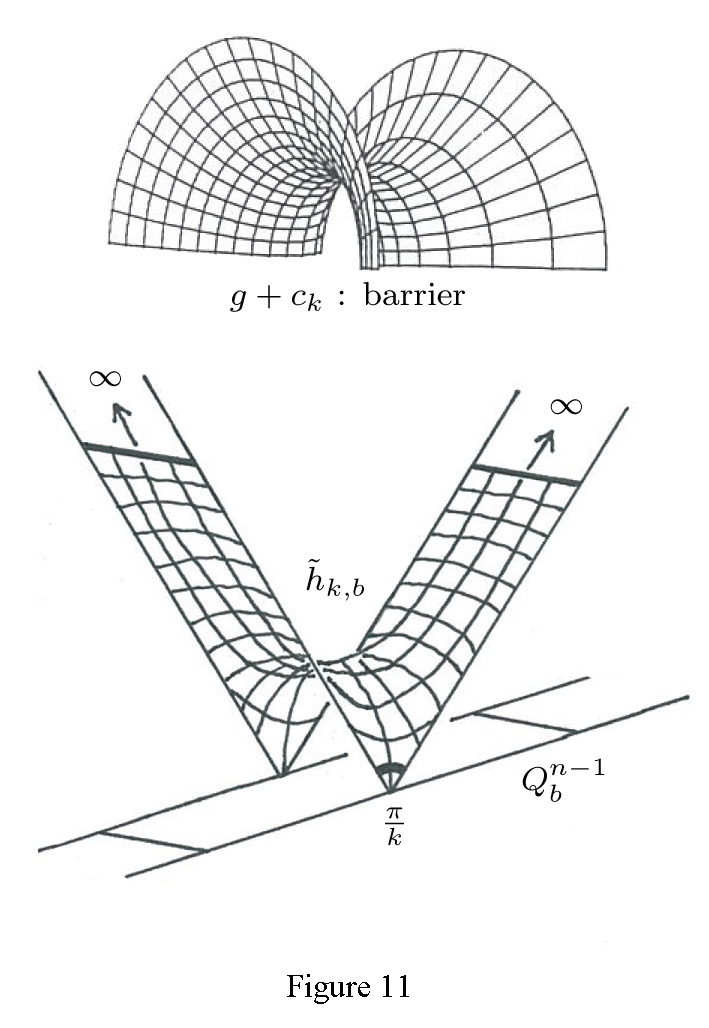}\\
\end{center}
Therefore the limit $\tilde{h}_k$ of $\tilde{h}_{k,b}$ as $b\rightarrow\infty$ exists on $Q_1^{n-1}$ (see Figure 11) and
$$\tilde{h}_k\leq g+c_k\,\,\,\,{\rm on}\,\,\,Q_1^{n-1}.$$

\begin{center}
\includegraphics[width=2.3in]{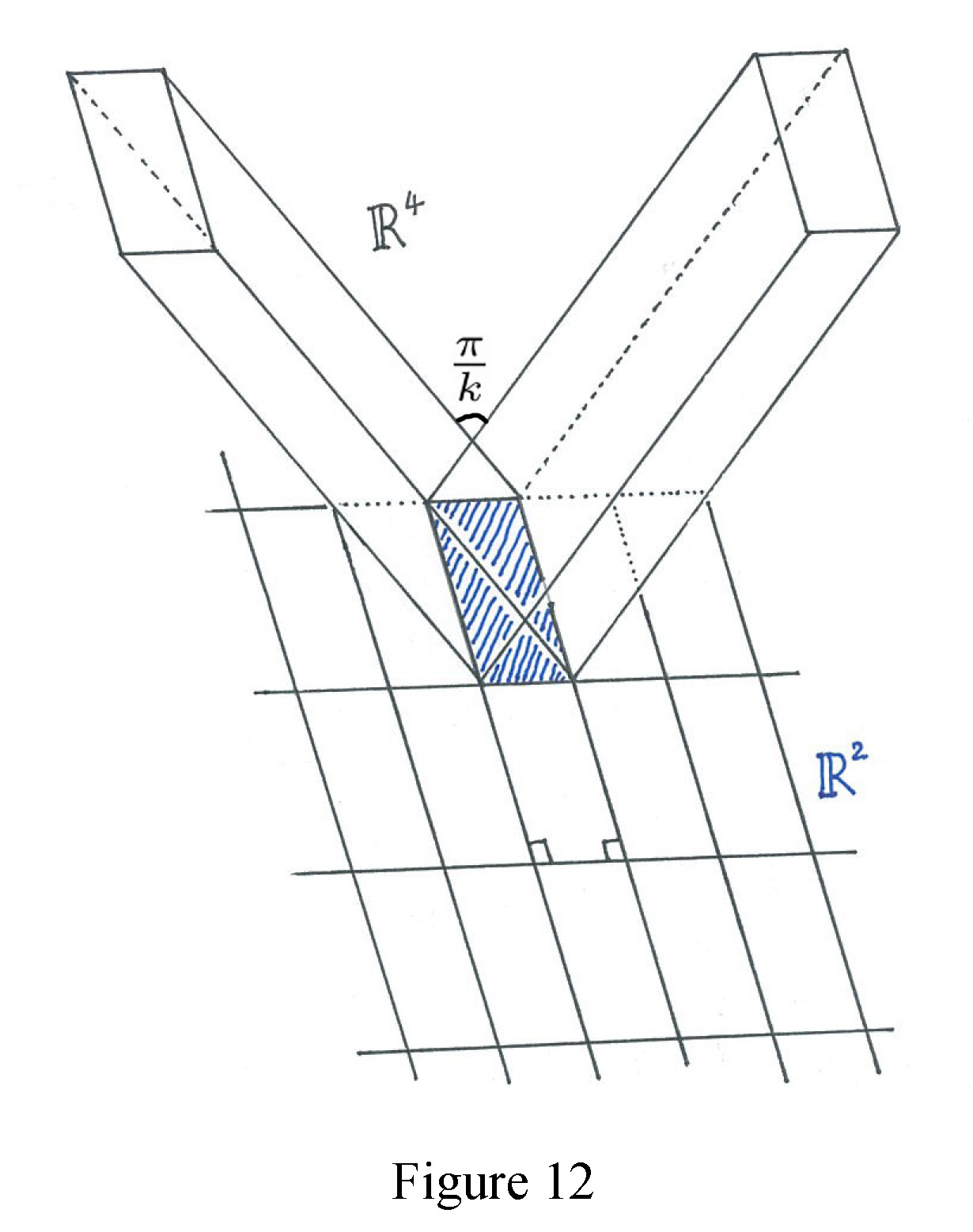}\\
\end{center}

We claim that $\tilde{h}_k$ exists on the infinite cube $Q_\infty^{n-1}$ as well. Note that
$$g\leq a\,\,\,\,{\rm on}\,\,\,\{0\}\times[0,a_2]\times\cdots\times[0,a_{n-1}].$$
Hence
$$\tilde{h}_{k,b}(x_1,\ldots,x_{n-1})\leq (a+c_k)+c_k|x_1|$$
on the boundaries of $[0,b]\times[0,a_2]\times\cdots\times[0,a_{n-1}]$ and $[-b,0]\times[0,a_2]\times\cdots\times[0,a_{n-1}]$ for any $b>0$.
Thus
$$\tilde{h}_{k,b}(x_1,\ldots,x_{n-1})\leq (a+c_k)+c_k|x_1|\,\,\,\,{\rm on}\,\,\,Q_b^{n-1}$$
for any $b$ and so $\tilde{h}_k$ exists on $Q_\infty^{n-1}$, as claimed. Clearly $\tilde{h}_k$ is analytic.

\begin{center}
\includegraphics[width=2in]{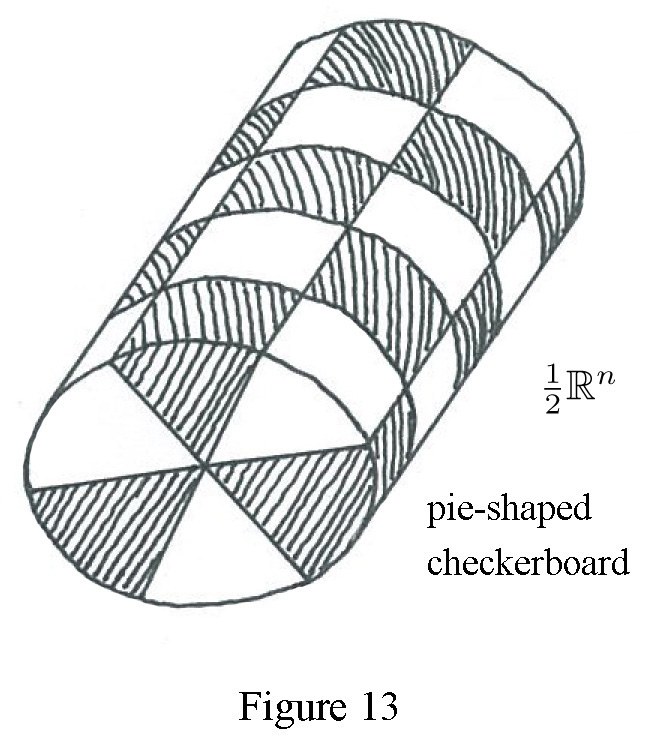}\\
\end{center}

Let $\Sigma_8$ be the graph of $x_n=\tilde{h}_{k}(x_1,\ldots,x_{n-1})$ on $Q_\infty^{n-1}$. $\Sigma_8$ inherits all the symmetries of $Q_\infty^{n-1}$, that is, $\Sigma_8$ is symmetric with respect to the $n-1$ vertical pairwise orthogonal hyperplanes of $\mathbb R^n$ which divide each interval of $Q_b^{n-1}$ into halves. Given an $(n-2)$-dimensional rectangular cube $Q^{n-2}$, let's call $Q^{n-2}\times\mathbb R^2$ an $(n-2)$-{\it slab} in $\mathbb R^n$. So a slab of $\mathbb R^3$ is called a 1-slab. The graph of the piecewise-linear function $x_n=h_k(x_1,\ldots,x_{n-1})$ divides the $(n-2)$-slab $[0,a_2]\times\cdots\times[0,a_{n-1}]\times (x_1x_n$-plane) into two components. The smaller one is (infinite) pie-shaped; let's denote it as $V$. As the two planar boundaries of $V$ in the interior of the $(n-2)$-slab make an angle of $\pi/k$, the $(n-2)$-slab $[0,a_2]\times\cdots\times[0,a_{n-1}]\times (x_1x_n$-plane) can be divided into $2k$ pie-shaped domains congruent to $V$. $\mathbb R^n$ can be divided into a tessellation $\mathcal{T}_0$ by $(n-2)$-slabs which are all congruent to $[0,a_2]\times\cdots\times[0,a_{n-1}]\times (x_1x_n$-plane) and one of which is $[0,a_2]\times\cdots\times[0,a_{n-1}]\times (x_1x_n$-plane) itself. One can refine $\mathcal{T}_0$ into another tessellation $\mathcal{T}_1$ by dividing each $(n-2)$-slab of $\mathcal{T}_0$ into $2k$ pie-shaped domains congruent to $V$. Let $\frac{1}{2}\mathbb R^n$ denote the union of all the pie-shaped domains in $\mathcal{T}_1$ which are chosen alternatingly such that $V\subset\frac{1}{2}\mathbb R^n$ (see Figure 13). $\frac{1}{2}\mathbb R^n$ is called the {\it pie-shaped checkerboard}.

$\Sigma_8$ is an embedded minimal hypersurface in $V$ so that $\partial\Sigma_8$ is a subset of the $(n-2)$-skeleton of $V$. Each pie-shaped domain $V_0$ of $\frac{1}{2}\mathbb R^n$ contains a unique minimal hypersurface $\Sigma_{V_0}$ which is congruent to $\Sigma_8$ and whose boundary is a subset of the $(n-2)$-skeleton of $V_0$. Define
$$\Sigma_S=\bigcup_{V_0\subset\frac{1}{2}\mathbb R^n}\Sigma_{V_0}.$$
Let $V_1,V_2$ be two neighboring pie-shaped domains of $\frac{1}{2}\mathbb R^n$ which share a nonempty subset $K$ of their $(n-2)$-skeletons. Then $\Sigma_{V_1}$ is the $180^\circ$-rotation of $\Sigma_{V_2}$ around $K$ because of the symmetries of $\Sigma_8$. Therefore $\Sigma_S$ is a complete embedded analytic minimal hypersurface as described by a), b), c).
\end{proof}

{\bf Remark 3.} The catenoid can be used as a barrier to construct even Scherk's second surface and the saddle towers in $\mathbb R^3$ without appealing to the Weierstrass representation formula. Moreover, H. Karcher's helicoidal saddle towers \cite{K} can be constructed in this way: Deform $V$ into $V_0$ which is invariant under a screw motion rotating around the $x_2$-axis and tessellate $\mathbb R^3$ by the domains congruent to $V_0$; use the half catenoid as a barrier to construct a minimal surface $\Sigma_{V_0}$ whose boundary is a subset of the 1-skeleton of $V_0$; keep rotating $\Sigma_{V_0}$ around its boundaries by $180$ degrees.

{\bf Remark 4.} Higher dimensional Scherk's first surface $\Sigma_{S_1}$ can be also constructed in $\mathbb R^n$ by solving the Dirichlet problem on the domain $O\cone F$, where $F$ is a face of the cube $[-1,1]^{n-1}\subset\mathbb R^{n-1}$. But $\Sigma_{S_1}$ has a self-intersection in case $n\geq4$ because the tessellation of $\mathbb R^{n-1}$ by the domains congruent to $O\cone F$ cannot generate the pyramid-shaped checkerboard $\frac{1}{2}\mathbb R^{n-1}$.

\end{document}